\documentclass[12pt,a4paper]{amsart}
\usepackage{url}
\usepackage{amsmath,amsthm,amsfonts,amssymb,latexsym}

\newtheorem{theorem}{Theorem}[section]
\newtheorem{proposition}[theorem]{Proposition}
\newtheorem{lemma}[theorem]{Lemma}
\newtheorem{claim}[theorem]{Claim}
\newtheorem{corollary}[theorem]{Corollary}
\theoremstyle{definition}
\newtheorem{remark}[theorem]{Remark}
\newtheorem{question}[theorem]{Question}

\title[Menger's property and $D$-spaces]{On the Menger covering property and $D$-spaces}

\author{Du\v{s}an Repov\v{s} and  Lyubomyr Zdomskyy}

\address{Faculty of Mathematics and Physics, and Faculty of Education,
University of Ljubljana, P. O. Box 2964, Ljubljana, Slovenia 1001.}
\email{dusan.repovs@guest.arnes.si}
\urladdr{http://www.fmf.uni-lj.si/\~{}repovs/index.htm}

\address{Kurt G\"odel Research Center for Mathematical Logic,
University of Vienna, W\"ahringer Stra\ss e 25, A-1090 Wien,
Austria.} \email{lzdomsky@gmail.com}
\urladdr{http://www.logic.univie.ac.at/\~{}lzdomsky/}

\thanks{The first author was supported by SRA grants P1-0292-0101 and J1-2057-0101.
The second author   acknowledges the support of FWF grant
P19898-N18. We would also like to thank Leandro Aurichi, Franklin
Tall, and Hang Zhang for kindly making their recent papers available
to us.}

\begin{document}

\begin{abstract}
The main results of this note are:
\begin{itemize}
\item It is consistent that every subparacompact space $X$ of size
$\omega_1$ is a $D$-space.
\item If there exists a Michael space, then all productively Lindel\"of spaces
have the Menger property, and, therefore, are $D$-spaces.
\item Every locally $D$-space which admits a $\sigma$-locally finite
cover by Lindel\"of spaces  is a $D$-space.
\end{itemize}
\end{abstract}

\maketitle

\section{Introduction}

A \emph{neighbourhood assignment} for a topological space $X$ is a
function $N$ from $X$ to the topology of $X$ such that $x\in N(x)$
for all $x$. A topological space $X$ is said to be a
\emph{$D$-space} \cite{vDaPfe79}, if for every neighbourhood
assignment $N$ for $X$ there exists a closed and discrete subset
$A\subset X$ such that $N(A)=\bigcup_{x\in A}N(x)=X$.

It is unknown whether paracompact (even Lindel\"of) spaces are
$D$-spaces. Our first result in this note answers
\cite[Problem~3.8]{Gru??} in the affirmative and may be thought of
as a very partial solution to this problem\footnote{While completing
this manuscript we have learned that this result has been
independently obtained by Hang Zhang and Wei-Xue Shi, see
\cite{ShiZha??}.}.

Our second result shows that the affirmative answer to
\cite[Problem~2.6]{Tal??}, which  asks whether all productively
Lindel\"of spaces are $D$-spaces, is consistent. It is worth
mentioning that our premises  (i.e., the existence of  a Michael
space) are not known to be inconsistent.

Our third result is a common generalization of two theorems from
\cite{MarSou09}.

Most of our proofs use either the recent important result of Aurichi
\cite{Aur10} asserting that every topological space with the Menger
property  is a $D$-space, or the ideas from its proof. We consider
only regular topological spaces. For the definitions of small
cardinals $\mathfrak d$ and $\mathrm{cov}(\mathcal M)$ used in this
paper we refer the reader to \cite{Vau90}.

\section{Subparacompact spaces of size $\omega_1$}

Following  \cite{BukHal03} we say that a topological space $X$ has
the property $E^*_\omega$ if for every sequence $\langle
u_n:n\in\omega\rangle$ of \emph{countable} open covers of $X$ there
exists  a sequence $\langle v_n:n\in\omega\rangle$ such that $v_n\in
[u_n]^{<\omega}$ and $\bigcup_{n\in\omega}\cup v_n =X$. In the realm
of Lindel\"of spaces the property $E^*_\omega$  is   usually called
the \emph{Menger property} or $\bigcup_{\mathit{fin}}(\mathcal
O,\mathcal O)$, see \cite{Tsa07} and references therein.

We say that a topological space $X$ has  property $D_\omega$, if for
every neighbourhood assignment $N$ there exists a countable
collection $\{A_n:n\in\omega\}$ of closed discrete subsets of $X$
such that $X=\bigcup_{n\in\omega}N(A_n)$. Observe that the property
$D_\omega$ is inherited by all closed subsets.

The following theorem is the main result of this section.

\begin{theorem} \label{aur}
Suppose that a topological space $X$ has  properties $D_\omega$ and
$E^*_\omega$. Then $X$ is a $D$-space.
\end{theorem}

The proof of Theorem~\ref{aur} is analogous to the proof of
\cite[Proposition~2.6]{Aur10}. In particular, it uses the following
 game  of length $\omega$ on a topological space $X$:
 On the $n$th  move player $I$ chooses a countable open cover
$u_n=\{U_{n,k}:k\in\omega\}$ such that $U_{n,k}\subset U_{n,k+1}$
for all $k\in\omega$, and player $II$ responds by choosing a natural
number $k_n$. Player $II$ wins the game if $\bigcup_{n\in\omega}
U_{n,k_n}=X$. Otherwise, player $I$ wins.  We shall call this game
an $E^*_\omega$-game. In the realm of Lindel\"of spaces this game is
known under the name \emph{Menger game}. It is well-known that a
Lindel\"of space $X$ has the  property $E^*_\omega$ if and only if
the first player has no winning strategy in the $E^*_\omega$-game on
$X$, see \cite{Hur25, COC1}. The proof of \cite[Theorem~13]{COC1}
also works  without any change for non-Lindel\"of spaces.

\begin{proposition} \label{meng_game}
A topological space $X$ has the property $E^*_\omega$ if and only if
the first player has no winning strategy in the  $E^*_\omega$-game.
\end{proposition}

A strategy of the first player in the $E^*_\omega$-game may be
thought of as a map $\Upsilon:\omega^{<\omega}\to\mathcal O(X)$,
where $\mathcal O(X)$ stands for the collection of all countable
open covers of $X$. The strategy $\Upsilon$ is winning, if
$X\neq\bigcup_{n\in\omega}U_{z\upharpoonright n, z(n)}$ for all
$z\in\omega^\omega$, where
$\Upsilon(s)=\{U_{s,k}:k\in\omega\}\in\mathcal O(X)$.

We are in a position now to present the proof of Theorem~\ref{aur}.
\begin{proof}
 We shall define a strategy $\Upsilon:X\to\mathcal
O(X)$ of the player $I$ in the $E^*_\omega$-game on $X$ as follows.
Set $F_\emptyset=X$.
 The property $D_\omega$ yields an increasing sequence $\langle
A_{\emptyset,k}:k\in\omega\rangle$ of closed discrete subsets of
$F_\emptyset$ such that $X=\bigcup_{k\in\omega}N(A_{\emptyset,k})$.
Set
$\Upsilon(\emptyset)=u_\emptyset=\{N(A_{\emptyset,k}):k\in\omega\}$.

Suppose that for some $m\in\omega$ and all $s\in\omega^{\leq m}$ we
have already defined a closed subset $F_s$ of $X$, an increasing
sequence $\langle A_{s,k}:k\in\omega\rangle $ of closed discrete
subsets of $F_s$, and a countable open cover $\Upsilon(s)=u_s$ of
$X$ such that $u_s=\{(X\setminus F_s)\cup N(A_{s,k}):k\in\omega\}$.

Fix  $s\in\omega^{m+1}$. Since $X$ has the property $D_\omega$, so
does its closed subspace $F_{s}:=X\setminus\bigcup_{i<m+1}
N(A_{s\upharpoonright i,s(i)})$, and hence there exists an
increasing sequence $\langle A_{s,k}:k\in\omega\}$ of closed
discrete subsets of $F_{s}$ such that
$F_{s}\subset\bigcup_{k\in\omega}N(A_{s,k})$. Set
$\Upsilon(s)=u_{s}=\{(X\setminus F_s)\cup N(A_{s,k}):k\in\omega\}$.
This completes the definition of $\Upsilon$.

Since $X$ has the property $E^*_\omega$, $\Upsilon$ is not winning.
Thus there exists $z\in\omega^\omega$ such that $X
=\bigcup_{n\in\omega} (X\setminus F_{z\upharpoonright n})\cup
N(A_{z\upharpoonright n,z(n)}).$ By the inductive construction,
$X\setminus F_{\emptyset}=\emptyset$ and $X\setminus
F_{z\upharpoonright n}=\bigcup_{i<n} N(A_{z\upharpoonright i,z(i)})$
for all $n>0$. It follows from above  that
 $X =\bigcup_{n\in\omega} N(A_{z\upharpoonright n,z(n)}).$
In addition, $ A_{z\upharpoonright n,z(n)}\subset
F_{z\upharpoonright n}=X\setminus \bigcup_{i<n}
N(A_{z\upharpoonright i,z(i)})$ for all $n>0$, which  implies that
$A:=\bigcup_{n\in\omega}A_{z\upharpoonright n,z(n)}$ is a closed
discrete subset of $X$. It suffices to note that  $N(A)=X$. \hfill
\end{proof}

We recall from \cite{Bur84} that a topological space $X$ is called
\emph{subparacompact}, if every open cover of $X$ has a
$\sigma$-locally finite closed refinement.

\begin{lemma} \label{l_subpara}
Suppose that $X$ is a  subparacompact topological space which can be
covered by $\omega_1$-many  of its Lindel\"of subspaces. Then $X$
has the property $D_\omega$.\footnote{By the methods of
\cite{ShiZha??} the submetalindel\"ofness is sufficient here.}

In particular, every subparacompact space of size $\omega_1$ has the
property $D_\omega$.
\end{lemma}
\begin{proof}
Let $\mathcal L=\{L_\xi:\xi<\omega_1\}$ be an increasing cover of
$X$ by Lindel\"of subspaces,
  $\tau$ be the topology of $X$, and
 $N:X\to\tau$ be a neighbourhood assignment. Construct by induction a sequence
$\langle C_\alpha:\alpha<\omega_1\rangle$ of (possibly empty)
countable subsets of $X$ such that
\begin{itemize}
 \item[$(i)$] $L_0\subset N(C_0)$;
\item[$(ii)$] $C_\alpha\cap N(\bigcup_{\xi<\alpha}C_\xi)=\emptyset$ for all $\alpha<\omega_1$; and
\item[$(iii)$] $L_\alpha\setminus N(\bigcup_{\xi<\alpha}C_\xi)\subset N(C_\alpha)$ for all $\alpha<\omega_1$.
\end{itemize}

 Set
$C=\bigcup_{\alpha<\omega_1}C_\alpha$. The subparacompactness of $X$
yields a closed cover  $\mathcal
F=\bigcup_{n\in\omega}\mathcal{F}_n$ of $X$ which refines $\mathcal
 U=\{N(x):x\in C\}$ and such that each $\mathcal{F}_n$ is
locally-finite. Since every element of $\mathcal{U}$ contains at
most countably many elements of $C$, so do elements of $\mathcal F$.
Therefore for every $F\in\mathcal F_n$ such that $C\cap
F\neq\emptyset$ we can write this intersection in the form
$\{x_{n,F,m}:m\in\omega\}$. Now it is easy to see that
$A_{n,m}:=\{x_{n,F,m}:F\in\mathcal F_n,C\cap F\neq\emptyset\}$ is a
closed discrete subset of $X$ and $\bigcup_{n,m\in\omega}A_{n,m}=C$.
\end{proof}

\begin{remark} What we have actually used in the proof of
Lemma~\ref{l_subpara} is the following weakening of
subparacompactness: every open cover $\mathcal{U}$ which is closed
under unions of its countable subsets admits a $\sigma$-locally
finite closed refinement. We do not know whether this property is
strictly weaker than  subparacompactness.
\end{remark}

\begin{corollary} \label{cor_tightness}
Let $X$ be a countably tight subparacompact topological space of
density $\omega_1$. Then $X$ has the property $D_\omega$.
\end{corollary}
\begin{proof}
Let $\{x_\alpha:\alpha<\omega_1\}$ be a dense subspace of $X$. Since
$X$ has countable tightness,
$X=\bigcup_{\alpha<\omega_1}\overline{\{x_\xi:\xi<\alpha\}}$. It
suffices to note that the closure of any countable subspace of a
subparacompact space is Lindel\"of.
\end{proof}

It is well-known  \cite[Theorem~4.4]{COC2}  (and it easily follows
from corresponding definitions) that any Lindel\"of space of size
$<\mathfrak d$ has the Menger property. The same argument shows that
every topological space of size $<\mathfrak d$  has the property
$E^*_\omega$. Combining this with  Theorem~\ref{aur} and
Lemma~\ref{l_subpara} we get the following corollary, which implies
the first of the results mentioned in our abstract.

\begin{corollary} \label{main_cor}
Suppose that $X$ is a  subparacompact topological space of size
$|X|<\mathfrak d$ which can be covered  by $\omega_1$-many  of its
Lindel\"of subspaces. Then $X$ is a $D$-space.
\end{corollary}

\section{Concerning the existence of a Michael space}

A topological space $X$ is said to be \emph{productively
Lindel\"of}, if $X\times Y$ is Lindel\"of for all Lindel\"of spaces
$Y$. It was asked in \cite{Tal??} whether  productively Lindel\"of
spaces are $D$-spaces.  The positive answer to the above question
has been proved consistent and in a stream of recent papers (see the
list of references in \cite{Tal??}) several sufficient
set-theoretical conditions were established. The following statement
gives a uniform proof for some of these results. In particular, it
implies \cite[Theorems~5 and 7]{Tal10},
\cite[Corollary~4.5]{AlaAurJunTal??},
 and answers \cite[Question~15]{Tal???} in the affirmative.

 A Lindel\"of space $Y$ is called a
\emph{Michael} space, if $\omega^\omega\times Y$ is not Lindel\"of.

\begin{proposition} \label{obs1}
If there exists a Michael space, then every productively Lindel\"of
space has the Menger property.
\end{proposition}

We refer the reader to \cite{Moo99} where the existence of a Michael
space was reformulated in a combinatorial language  and a number of
set-theoretic conditions guaranteeing the existence of Michael
spaces were established.

In the proof of Proposition~\ref{obs1} we shall use set-valued maps,
see \cite{RepSem98}. By a \emph{set-valued map} $\Phi$ from a set
$X$ into a set $Y$ we understand a map from $X$ into $\mathcal P(Y
)$ and write $\Phi : X\Rightarrow Y$  (here $\mathcal P(Y)$ denotes
the set of all subsets of $Y$). For a subset $A$ of $X$ we set
$\Phi(A) = \bigcup_{x\in A}\Phi(x) \subset Y$. A set-valued map
$\Phi$ from a topological spaces $X$ to a topological space $Y$ is
said to be
\begin{itemize}
\item  \emph{compact-valued}, if $\Phi(x)$  is compact for every $x \in X$;
\item \emph{upper semicontinuous}, if for every open subset $V$ of $Y$ the
set $\Phi^{-1}_\subset (V) = \{x \in  X : \Phi(x) \subset V\}$ is
open in $X$.
\end{itemize}

The proof of the following claim is straightforward.

\begin{claim} \label{trivial}
\begin{enumerate}
 \item Suppose that $X,Y$ are topological spaces, $X$  is  Lindel\"of,
 and  $\Phi:X\Rightarrow Y$
is a compact-valued upper semicontinuous map such that $Y=\Phi(X)$.
Then $Y$ is Lindel\"of.
\item If $\Phi_0:X_0\Rightarrow Y_0$ and $\Phi_0:X_1\Rightarrow Y_1$
are compact-valued upper semicontinuous, then so is the map
$\Phi_0\times\Phi_1 :X_0\times X_1\Rightarrow Y_0\times Y_1$
assigning to each $(x_0,x_1)\in X_0\times X_1$ the product
$\Phi_0(x_0)\times\Phi_1(x_1)$.
\end{enumerate}
\end{claim}

\noindent\textit{Proof of Proposition~\ref{obs1}.} Suppose, contrary
to our claim, that $X$ is a productively Lindel\"of
 space which does not have the
 Menger property and $Y$ is a Michael space.
 It suffices to show that $X\times Y$ is not Lindel\"of.

 Indeed, by \cite[Theorem~8]{Zdo05}
 there exists a compact-valued upper semicontinuous map
 $\Phi:X\to \omega^\omega$ such that $\Phi(X)=\omega^\omega$.
By Claim~\ref{trivial}(2) the product $\omega^\omega\times Y$ is the
image of $X \times Y$ under a compact-valued upper semicontinuous
map. By the definition of a Michael space,
 $\omega^\omega\times Y$ is not Lindel\"of. By applying Claim~\ref{trivial}(1) we can
conclude that $X\times Y$ is not Lindel\"of neither. \hfill $\Box$
\medskip

By a result of Tall \cite{Tal10} the existence of a Michael space
implies that all productively Lindel\"of analytic metrizable spaces
are $\sigma$-compact. Combining recent results obtained in
\cite{AlaAurJunTal??} and \cite{Rep??}
  we can consistently extend this result
to all $\Sigma^1_2$ definable subsets of $2^\omega$.

\begin{theorem} \label{sigma_1_2}
Suppose that $\mathrm{cov}(\mathcal M)>\omega_1$ and there exists a
Michael space. Then every productively Lindel\"of $\Sigma^1_2$
definable subset of $2^\omega$ is $\sigma$-compact.
\end{theorem}
\begin{proof}
Let $X$ be a productively Lindel\"of $\Sigma^1_2$ definable subset
of $2^\omega$.

If $X$ cannot be written as a union of $\omega_1$-many of its
compact subspaces, then it contains a closed copy of $\omega^\omega$
\cite{Rep??}, and hence the existence of the Michael space implies
that $X$ is not productively Lindel\"of, a contradiction.

Thus $X$ can be written as a union of $\omega_1$-many of its compact
subspaces, and therefore it is $\sigma$-compact by
\cite[Corollary~4.15]{AlaAurJunTal??}.
\end{proof}

We do not know whether the assumption $\mathrm{cov}(\mathcal
M)>\omega_1$ can be dropped from Theorem~\ref{sigma_1_2}.

\begin{question}
Suppose that there exists a Michael space. Is every coanalytic
productively Lindel\"of space $\sigma$-compact?
\end{question}

By \cite[Proposition~31]{Tal????} the affirmative answer to the
question above follows from the Axiom of Projective Determinacy.

\section{Locally finite unions}

\begin{theorem} \label{simp_obs}
Suppose that $X$ is a locally $D$-space which admits a
$\sigma$-locally finite cover by Lindel\"of spaces. Then $X$ is a
$D$-space.
\end{theorem}
\begin{proof}
Let $\mathcal F=\bigcup_{n\in\omega}\mathcal F_n$ be a cover of $X$
by Lindel\"of subspaces such that $\mathcal F_n$ is locally finite.
Fix $F\in\mathcal F_n$. For every $x\in F$ there exists an open
neighbourhood $U_x$ of $x$ such that $\bar{U_x}$ is a $D$-space. Let
$C_F$ be a countable subset of $F$ such that $F\subset\bigcup_{x\in
C_F}U_x$. Then $\mathcal Z_F=\{\overline{F\cap U_x}:x\in C_F\}$ is a
countable cover of $F$ consisting of closed $D$-subspaces of $X$
such that $F\cap Z$ is dense in $Z$ for all $Z\in\mathcal Z_F$. It
follows from the above that $X$ admits a $\sigma$-locally finite
cover consisting of closed $D$-subspaces. Since a union of a locally
finite family of closed $D$-subspaces is easily seen to be a closed
$D$-subspace, $X$ is a union of an
 increasing sequence of its closed $D$-subspaces. Therefore it is a $D$-space
 by results of \cite{BorWeh91}.
\end{proof}

\begin{corollary}\label{cor3_1}
 If a topological space $X$  admits a $\sigma$-locally finite locally countable
 cover by topological spaces with the Menger property,
then it is a $D$-space.

In particular, a locally Lindel\"of space admitting  a
$\sigma$-locally finite cover by topological spaces with the Menger
property is a $D$-space.
\end{corollary}
\begin{proof}
The second part is a direct consequence of the first one since every
$\sigma$-locally countable family of subspaces of a locally
Lindel\"of space is locally countable.

To prove the first assertion, note that by local countability  every
point $x\in X$ has a closed neighbourhood which is a countable union
of its  subspaces  with the Menger property, and hence it has  the
Menger property itself. Therefore $X$ is a locally $D$-space. It now
suffices to apply Theorem~\ref{simp_obs}.
\end{proof}

It is known that every Lindel\"of $\mathcal C$-scattered space  is
$\mathcal C$-like, and that $\mathcal C$-like spaces have the Menger
property, see \cite[p.~247]{Tel87} and references therein. Thus
Corollary~\ref{cor3_1} implies Theorems~2.2 and 3.1 from
\cite{MarSou09}.



\bibliographystyle{amsplain}

\end{document}